\documentclass{birkjour}
\usepackage{amsmath,amsthm,amssymb}

\theoremstyle{plain}
\newtheorem{theorem}{Theorem}

\newtheorem{proposition}[theorem]{Proposition}

\theoremstyle{remark}

\newtheorem{remark}[theorem]{Remark}
\newtheorem{example}[theorem]{Example}
\numberwithin{theorem}{section}

\newcommand{\F}{\mathbb{F}}

\newcommand{\Q}{\mathbb{Q}}
\newcommand{\Z}{\mathbb{Z}}

\newcommand{\Tr}{\textnormal{Tr}}

\newcommand{\Hom}{\textnormal{Hom}}

\newcommand{\End}{\textnormal{End}}

\newcommand{\mykill}[1]{}

\begin{document}

\bibliographystyle{plain}

\title{On  free elementary $\Z_p C_p$-lattices.}
\author{Gabriele Nebe}
\email{nebe@math.rwth-aachen.de}
\address{Lehrstuhl f\"ur Algebra und Zahlentheorie, RWTH Aachen University, 52056 Aachen, Germany}

\begin{abstract}
We show that all elementary lattices that 
are free $\Z_p C_p$-modules admit an orthogonal decomposition 
	into a sum of  free unimodular and $p$-modular $\Z_p C_p$ lattices.
	\\
	MSC: 11H56; 11E08 \\
	{\sc keywords:} quadratic forms over local rings; automorphism groups of lattices; free modules; Jordan decomposition; Smith normal form;\\
\end{abstract}

\maketitle

\section{Introduction}

Let $R := \Z_p C_p$ denote the group ring of the cyclic group of
order $p$ over the localisation of $\Z $ at the prime $p$. 
The present paper considers free $R$-lattices
$L\cong R^a$. 
The main observation in this situation is 
Theorem \ref{eltdiv}: 
Given two free $R$-modules $M$ and $L$ 
with 
$pM \subseteq L \subseteq M $ 
then there is an $R$-basis $(g_1,\ldots , g_a)$ of $M$ and $0\leq t\leq a$
such that $(g_1,\ldots , g_t, p g_{t+1},\ldots , p g_a )$ is an 
$R$-basis of $L$. 
So these lattices do admit a compatible basis. 
Applying this observation to Hermitian $R$-lattices shows
that free elementary Hermitian $R$-lattices admit an 
invariant splitting (see Theorem \ref{pmodautpfree}) 
as the orthogonal sum of a free unimodular lattice and a free $p$-modular
lattice.

The results of this note have
been used in the thesis \cite{Eisenbarth} to 
study extremal lattices admitting an automorphism of 
order $p$ in the case that $p$ divides the level
of the lattice.

\section{Existence of compatible bases} 

For a prime $p$ we denote by 
$$\Z_{p} := \{ \frac{a}{b} \in \Q \mid  p \mbox{ does not divide }  b \} $$
	the localisation of $\Z $ at the prime $p$. 
	The following arguments also apply accordingly 
	to the completion of this discrete valuation ring.
Let  $R:=\Z _{p} C_p$ denote the 
group ring of the cyclic group $C_p = \langle \sigma \rangle $ 
of order $p$. 
Then $e_1:=\frac{1}{p} (1+\sigma + \ldots + \sigma ^{p-1}) \in \Q C_p $
and $e_{\zeta }:=1-e_1$ are the primitive idempotents in the 
group algebra $\Q C_p$ with $\Q C_p = \Q C_p e_1 \oplus \Q C_p e_{\zeta } 
\cong  \Q \oplus \Q [\zeta _p]$, where $\zeta _p$ is a primitive $p$-th 
root of unity. The ring $T:=\Z_p[\zeta_p]$ is a discrete valuation 
ring in the $p$-th cyclotomic field $\Q[\zeta _p]$ with 
prime element $\pi := (1-\zeta _p)$ and hence 
$$Re_1 \oplus R e_{\zeta } \cong \Z _{p} \oplus \Z _{p} [\zeta _p] =: S \oplus T $$ 
is the unique maximal $\Z _p$-order in $\Q C_p$. 

\begin{remark} \label{structR}
With the notation above $T/(\pi ) \cong \Z_p / (p) \cong \F_p $ and 
via this natural ring epimorphism 
$$ R = \{ (x,y) \in  \Z _{p} \oplus \Z _{p} [\zeta _p] \mid 
x + p \Z_p = y + \pi \Z _{p} [\zeta _p]  \} .$$
$R$ is generated as $\Z_p$-algebra by 
$1=(1,1)$ and $1-\sigma = (0,\pi )$. 
Moreover $Re_1 \cap R = p R e_1 = p S$ and 
$Re_{\zeta } \cap R = \pi R e_{\zeta } = \pi T$ and 
	the radical $J(R) := pS \oplus \pi T$ of $R$ 
	is the unique maximal ideal of the
local ring $R$.
\end{remark} 

By \cite{Reiner} the indecomposable $R$-lattices are the free $R$-module $R$, 
the trivial $R$-lattice $\Z_p = Re_1=:S$ and 
the lattice $\Z_{p} [\zeta _p]=Re_{\zeta} =:T$
in the rational irreducible faithful representation of $C_p$. 
The theorem by Krull-Remak-Schmidt-Azumaya \cite[Chapter 1, Section 11]{FeitBook} ensures that 
any finitely generated 
$R$-lattice $L$ is a direct sum of indecomposable $R$-lattices
$$L \cong R^a \oplus T^b \oplus S^c .$$

In this note we focus on the case of free $R$-lattices. 
Though $R$ is not 
a principal ideal domain,
for certain sublattices of free $R$-lattices
there do exist compatible bases:

\begin{theorem} \label{eltdiv}
Let $M\cong R^a$ be a free $R$-lattice of rank $a$.
Assume that $L$ is a free $R$-lattice with $pM \subseteq L \subseteq M$. 
Then there is an $R$-basis $(g_1,\ldots , g_a)$ of $M=Rg_1\oplus \ldots \oplus Rg_a$
and $0\leq t\leq a$ such that 
$$L = Rg_1\oplus \ldots \oplus Rg_t \oplus pRg_{t+1} \oplus \ldots \oplus p R g_a .$$
\end{theorem} 

\begin{proof}
	Let $\tilde{S} := M e_1 $ and $\tilde{T} := M e_{\zeta }$. 
	Now $M \cong R^a$ is a free $R$-lattice, so, as in Remark \ref{structR},
	$M $ is a sublattice of 
	$\tilde{S}\oplus \tilde{T}$ of index $p^a$,
	$\tilde{S} \cap M = p \tilde{S} $, and 
	$\tilde{T} \cap M = \pi \tilde{T} $.
	The Jacobson radical is 
	$J(M) = J(R) M = p\tilde{S} \oplus \pi \tilde{T} $ 
	and of index $p^a$ in $M$.
We proceed by induction on $a$. \\
	If $a=1$, then $M=R$, $\tilde{S} \cong S$, $\tilde{T} \cong T$. 
As $M/pM \cong \F_pC_p  \cong \F_p[x]/(x-1)^p $ is a chain ring, 
the $R$-sublattices of $M$ that contain $pM$ form a chain: 
	$$M \supset
	p \tilde{S} \oplus \pi \tilde{T} 
\supset 
	p \tilde{S} \oplus \pi^2 \tilde{T} 
\supset \ldots \supset 
	p \tilde{S} \oplus \pi^{p-2} \tilde{T} \supset p \tilde{S} \oplus p \tilde{T} 
  \supset  pM .$$
The only free $R$-lattices among these are $M$ and $pM$.  \\
Now assume that $a>1$.
If $ L \not\subseteq J(M)$ then we may choose $g_1 \in L\setminus J(M)$.
As $g_1\not\in J(M)$ the $R$-submodule $Rg_1$ of $M$ is a free submodule of 
both modules $L$ and $M$, so $M=Rg_1 \oplus M'$, $L=Rg_1 \oplus L'$ where
$M'$ and $L'=L\cap M'$ are free $R$-lattices of rank $a-1$ satisfying the 
assumption of the theorem and the theorem follows by induction. 
So we may assume that 
	\begin{equation}\label{LJM} 
L\subseteq J(M) = p\tilde{S} \oplus \pi \tilde{T} .
	\end{equation} 
The element $e_1\in \Q C_p$ is a central idempotent in 
	$\End_R(J(M))$ projecting onto $p\tilde{S}= J(M) e_1$.
The assumption
	that $pM \subseteq L \subseteq J(M)$ implies that 
	$$p\tilde{S} =  pM e_1 \subseteq  L e_1 \subseteq  J(M) e_1 = p\tilde{S} . $$ 
So $L e_1 = pM e_1 =p\tilde{S}$. 

	To show that $L=pM$ we first show that $Le_{\zeta } = pM e_{\zeta }$. \\
	As $pM\subseteq L$ we clearly have that $pMe_{\zeta} \subseteq L e_{\zeta} $. \\
	To see the opposite inclusion put $K :=  L\cap L e_{\zeta} $
	to be the kernel of the projection $e_1: L \to Le_1$.
	As $L$ is free, we get as in Remark \ref{structR} 
	that $K = \pi Le_{\zeta } $.
	Let $k $ be maximal such that $ K \subseteq \pi^{k} \tilde{T}$. 
	Then $k\geq 2$ because $Le_{\zeta } \subseteq \pi \tilde{T}$
	(see equation \eqref{LJM}).
\\
Assume that $k \leq p-1$. 
There is $\ell \in L$ such that $y=\ell e_{\zeta} \not\in \pi^{k} \tilde{T} $.
As $pMe_1 = Le_1$, there is $m \in M $ such that $pm e_1= \ell e_1$. 
Now $pM\subseteq L$ so $pm \in L$ and
	$\ell-pm \in K= Ke_{\zeta }$. 
	\\
	We compute $\ell -pm = (\ell -pm ) e_{\zeta } = y - pm e_{\zeta }$.
	\\
As $pMe_{\zeta } = p \tilde{T} = \pi^{p-1}\tilde{T}$ and 
 $y\not\in \pi^{k}\tilde{T}$ the assumption that $k\leq p-1$ 
	shows that  $\ell -pm \not \in \pi^k \tilde{T}$, 
which contradicts the definition of $k $.
\\
Therefore $k \geq p$ and $L e_{\zeta } \subseteq pM e_{\zeta } $.
	\\
	Now $pM$ and $L$ both have index $p^a$ in 
	$pMe_1 \oplus pM e_{\zeta } = Le_1 \oplus Le_{\zeta }$
	(again by Remark \ref{structR} as $L$ and $M$ are free).
	So the assumption $pM \subseteq L$ implies that $pM = L$.
\end{proof}

\begin{remark}
	Let $M\cong T^b\oplus S^c$ and let $L$ be a sublattice of $M$
	again isomorphic to $T^b \oplus S^c$.
	Then $M=Me_{\zeta } \oplus Me_1$ and 
	 $L=Le_{\zeta } \oplus Le_1$. 
	 By the main theorem for 
	modules over principal ideal domains there is a 
	$T$-basis $(x_1,\ldots , x_b)$ of $Me_{\zeta} $ 
	and an $\Z_p$-basis $(y_1,\ldots, y_c) $ of $Me_1$, 
	as well as $0\leq n_1\leq  \ldots \leq n_b$, 
	 $0\leq m_1\leq \ldots \leq m_c$,  such that 
	 $L = \bigoplus _{i=1}^b \pi^{n_i} T x_i \oplus 
	  \bigoplus  _{i=1}^c p^{m_i} \Z_p y_i  .$
\end{remark} 

\begin{example} 
	For general modules $M$, however, Theorem \ref{eltdiv} has no 
	appropriate analogue.
	To see this consider $M\cong R\oplus S $ and choose a pseudo-basis
	$(x,y)$ of $M$ such that $x$ generates a free direct summand
	and $y$ its complement  isomorphic to $S$. 
Let $L$ be the $R$-sublattice generated by $p x e_1$ and 
	$x(1-\sigma ) + y$. As $x(1-\sigma )+y$ generates a free 
	$R$-sublattice of $M$ and $R(pxe_1)\cong S$ we have 
	$L  \cong S \oplus R $.
	For $p>2$ we compute that $pM\subseteq L \subseteq M$.  Then the 
	fact that $|M/L|=p^2$ implies that 
	these two modules do not admit a compatible pseudo-basis. 
\end{example}

\section{Lattices in rational quadratic spaces} \label{autp}
From now on we consider $\Z_{p} $-lattices 
$L$ in a non-degenerate rational bilinear space $(V,B)$.
The {\em dual lattice} of $L$ is
$$L^{\#} := \{ x\in V \mid B(x,\ell ) \in \Z_{p}  \mbox{ for all } \ell \in L \}.$$
The lattice $L$ is called {\em integral}, if 
$L \subseteq L^{\#} $ and {\em elementary}, 
if $$pL^{\#} \subseteq L \subseteq L^{\#} .$$
Following O'Meara \cite[Section 82 G]{OMeara} we call a 
lattice $L$ {\em unimodular} if $L=L^{\#}$ and 
{\em $p^j$-modular} if $p^jL^{\#} = L$.

We now assume that $\sigma $ is an automorphism of 
$L$ of order $p$, so 
 $\sigma $ is an orthogonal mapping of $(V,B)$ with $L \sigma  = L$.
Then also the dual lattice
$L^{\#} $ is a $\sigma $-invariant lattice in $V$. 
As the dual basis of a lattice basis of $L$ is a lattice basis of 
$L^{\# }$, 
the bilinear form $B$ yields an identification between 
$L^{\#}$ and the lattice $\Hom _{\Z_p}(L,\Z_p)$ 
of $\Z_p$-valued linear forms on $L$. 
The $\sigma $-invariance of $B$ shows that this is an isomorphism of
$\Z_p[\sigma ]$-modules.

\begin{remark}
	As a $\Z_{p}[\sigma ]$-module we have
	$L^{\#} \cong \Hom _{\Z_{p}} (L,\Z_{p}) $.
\end{remark} 

As all indecomposable $\Z_{p}[\sigma ]$-lattices are 
isomorphic to their homomorphism lattices, we obtain

\begin{proposition}\label{dualtype}  (see \cite[Lemma 5.6]{PreprintNebe}) 
	If $L\cong R^a\oplus T^b \oplus S^c $ as $\Z_{p}[\sigma ]$-lattice
	then also $L^{\#}  \cong R^a\oplus T^b \oplus S^c $.
\end{proposition}

The group ring $R$ comes with a natural involution $\overline{\phantom{x}} $,
the unique $\Z_p$-linear map $\overline{\phantom{x}}:R \to R$
with $\overline{\sigma ^i} = \sigma ^{-i} $ for all $0\leq i\leq p-1$. 
 This involution 
is the restriction of the involution on 
the maximal order $S\oplus T$ that is trivial 
on $S$ and the complex conjugation on $T$. 

\begin{remark}\label{dual}
	The $\Z_p$-lattice $R$ is unimodular with respect to 
	the symmetric bilinear form 
	$$R\times R \to \Z_p , (x,y) \mapsto \frac{1}{p} \Tr _{reg} (x\overline{y} ) $$
	where $\Tr _{reg} : \Q C_p \to \Q $ denotes the regular trace of the
	$p$-dimensional $\Q $-algebra $\Q C_p$.
	We thus obtain a bijection between the 
	set of $\sigma $-invariant $\Z_p$-valued bilinear forms on 
	the $R$-lattice $L$ and the $R$-valued Hermitian forms on $L$:
	If $h:L\times L \to R $ is such a Hermitian form, then 
	$B=\frac{1}{p} \Tr _{reg} \circ h $ is a bilinear $\sigma $-invariant 
	form on $L$. As $R=R^{\#}$ these forms yield the 
	same notion of duality. 
In particular the dual lattice $L^{\# }$ of a 
free lattice $L = \oplus _{i=1}^a R g_i $ is again 
free  $L^{\#} = \oplus _{i=1}^a R g^*_i $ with the Hermitian dual basis 
	$(g^*_1,\ldots , g^*_a)$
as a lattice basis, giving a constructive argument for 
	Proposition \ref{dualtype} for free lattices.
\end{remark}

\section{Free elementary lattices}
\label{pautpmod} 
In this section we assume that $L$ is an 
elementary lattice and $\sigma$ 
an automorphism of $L$ of prime order $p$. 
Recall that $R$ is the commutative ring
$R:= \Z_{p} [\sigma ] $, so $L$ is an $R$-module.

\begin{theorem} \label{pmodautpfree}
	Let $p$ be a prime and let 
	$L$ be an elementary lattice with an automorphism 
	$\sigma $ such that 
	$L\cong R^a$ is a  free $R$-module.  Then also $L^{\#} \cong R^a$ and
there is an $R$-basis $(g_1,\ldots , g_a)$ 
	of $L^{\#}$ and $0\leq t \leq a$ such that 
$(g_1,\ldots , g_t, p g_{t+1},\ldots , p g_a ) $
is an  $R$-basis of $L $. 
In particular $L$ is the orthogonal sum of the 
unimodular  free $R$-lattice $L_0:=R g_1\oplus \ldots \oplus R g_t$ 
and a $p$-modular 
	free $R$-lattice $L_1 := L_0^{\perp }$.
\end{theorem}

\begin{proof}
	Under the assumption both lattices 
$L$ and $M:=L^{\#}$ 
 are free $R$-modules satisfying $pM\subseteq L\subseteq M$. 
 So by Theorem \ref{eltdiv} there is a basis 
 $(g_1,\ldots , g_a)$ of $M$ such that 
$(g_1,\ldots g_t, pg_{t+1}, \ldots , pg_a)$ is a basis of $L$. 
Clearly $L$ is an 
integral lattice and $L_0:=Rg_1\oplus \ldots \oplus Rg_t$ is a unimodular
sublattice of $L$. By \cite[Satz 1.6]{Kneser} unimodular free sublattices
split as orthogonal summands, so 
$L=L_0 \perp L_1$ with $L_1^{\#} = \frac{1}{p} L_1$, i.e. 
$L_1$ is $p$-modular.
\end{proof}

Note that the assumption that the lattice is elementary is necessary, as 
the following example shows.

\begin{example}\label{dim2} 
	Let $L=R g_1 \oplus R g_2 $ be a free lattice of rank 2 with 
$R$-valued Hermitian form defined by the Gram matrix 
	$$\left(\begin{array}{cc} (p,0) & (0,\pi ) \\
	(0,\overline{\pi}) & (p,0) \end{array} \right) .$$
	Here we identify $R$ as a subring of $S\oplus T$, so 
	$(p,0) = p e_1 = 1+\sigma + \ldots + \sigma ^{p-1} $ and 
	$(0,\pi) = (0,(1-\zeta_p)) = 1-\sigma \in R $. 
	Then $L$ is orthogonally indecomposable, because $Le_{\zeta }$
	is an orthogonally indecomposable $T$-lattice, but $L$ 
	is not modular. 
	Note that the base change matrix between $(g_1,g_2)$ and 
	the dual basis, an $R$-basis of  $L^{\#}$, is 
	the inverse of the Gram matrix above, so 
	$$\left(\begin{array}{cc} (p^{-1},0) & (0,-\overline{\pi }^{-1} ) \\ 
	(0,-{\pi}^{-1}) & (p^{-1},0) \end{array} \right) .$$
	As $(1,0) = e_1 \not\in R$ this shows that $pL^{\#} \not\subseteq L$,
	so $L$ is not an elementary lattice. 
	\end{example}

\bibliography{dim36}

\begin{thebibliography}{1}

\bibitem{Eisenbarth}
Simon Eisenbarth.
\newblock Gitter und {C}odes \"uber {K}ettenringen.
\newblock {\em Thesis, RWTH Aachen University}, 2020.

\bibitem{FeitBook}
Walter {Feit}.
\newblock {\em The representation theory of finite groups}.
\newblock North Holland, 1982.

\bibitem{Kneser}
M.~Kneser.
\newblock {\em Quadratische {F}ormen}.
\newblock Springer-Verlag, Berlin, 2002.
\newblock Revised and edited in collaboration with Rudolf Scharlau.

\bibitem{PreprintNebe}
G.~Nebe.
\newblock Automorphisms of modular lattices.
\newblock {\em Journal of Pure and Applied Algebra}, to appear.

\bibitem{OMeara}
O.~T. O'Meara.
\newblock {\em Introduction to {Q}uadratic {F}orms}.
\newblock Springer, 1973.

\bibitem{Reiner}
Irving Reiner.
\newblock Integral representations of cyclic groups of prime order.
\newblock {\em Proc. Amer. Math. Soc.}, 8:142--146, 1957.

\end{thebibliography}

\end{document}